\documentclass{amsart}
\usepackage{amssymb,amsmath,amsthm,epsf,epsfig,dsfont,bbm}
\usepackage{latexsym}
\usepackage[utf8]{inputenc}
\usepackage[backref=page]{hyperref}
\hypersetup{
	colorlinks   = true,
	citecolor    = blue,
	linkcolor    = red
	}
\usepackage{upref, eucal}
\usepackage[all]{xy}

\newcommand {\nc} {\newcommand}
\newcommand {\enm} {\ensuremath}

\def \d{\delta}
\nc {\nd}{\partial}

\nc {\bdm} {\begin{displaymath}}
\nc {\edm} {\end{displaymath}}

\newtheorem {theorem} {\bf{Theorem}}[section]
\newtheorem {lemma}[theorem] {\bf Lemma}
\newtheorem {proposition}[theorem] {\bf Proposition}

\newtheorem {remark}[theorem]{Remark}
\newtheorem {definition}[theorem]{\bf Definition}
\newtheorem {corollary}[theorem] {\bf Corollary}
\numberwithin {equation}{section}

\newcommand\EE{\mathbb{E}}
\newcommand\ZZ{\mathbb{Z}}

\newcommand{\Ou}{\enm{\mathcal{O}}}

\nc{\J}{\enm{\mathcal{J} }}
\nc {\Z} {\enm{\mathbb{Z}}}
\nc {\form}[1] {\enm{\mbox{\underline{for}}}_{#1}}
\nc {\prol}[1] {\enm{\mbox{\underline{prol}}_{{#1}^*}}}

\nc {\stk} {\stackrel}

\newcommand{\map}{\rightarrow}

\newcommand{\inj}{\hookrightarrow}

\newcommand{\dualmod}[1]{{#1}^{\vee}}

\newcommand{\Pn}[2] {\ensuremath{ {\mathbb{P}}^{#1}_{#2}}}
\nc{\Quot}[3]{\enm{ {\mathfrak{Quot}_{ {#1}/{#2}/{#3}}}}}
\nc{\Hilb}[2]{\enm{ {\mathfrak{Hilb}_{ {#1}/{#2}}}}}
\newcommand{\mfrak}[1]{\mathfrak{#1}}

\newcommand{\bb}[1]{\mathbb{#1}}
\newcommand{\mcal}[1]{\mathcal{#1}}

\newcommand{\Q}{\mathbb{Q}}
\nc {\Coh}[4] {\ensuremath{H^{#1}(\Pn{#2}{},{#3}({#4}))}}
\nc {\Ch}[3] {\enm{H^{#1}(X_t,{#2}_t({#3}))}}
\nc {\Qphi}[4]{\enm{ {\mathfrak{Quot}^{~#4}_{ {#1}/{#2}/{#3}}}}}
\nc {\Gra}[4]{\enm{ {\mathfrak{Grass}_{#2}({#3},{#4})}}}
\nc {\HomA}[2]{\enm{\mathrm{Hom}_A{#1}{#2}}}
\nc {\tr}{\mathrm{tr}}

\nc {\C}[2]{\enm{\left(\begin{array}{l} {#1} \\ {#2} \end{array} \right)}}
\nc {\mat}[4]{\enm{\left(\begin{array}{ll}{#1} & {#2} \\ {#3} & {#4}
\end{array}\right)}}

\def \mb{\mbox}

 \def \Z{{\mathbb Z}}  
\def \cF{\mathcal F}

  \def \cF{\mathcal F} \def \h{\hat{\ }}

\def \d{\delta}   \def \bF{{\bf F}}

\def \cO{\mathcal O}  \def \bX{{\bf X}} \def \bH{{\bf H}}
   \def \bF{{\mathbb F}}

\def \hG{\hat{\mathbb{G}}_{\mathrm{a}}}

\def \R1{R((q))[q']\h}

\usepackage{xcolor}

\DeclareMathOperator{\Spec}{\mathrm{Spec}}
\DeclareMathOperator{\Spf}{\mathrm{Spf}}
\DeclareMathOperator{\Lie}{\mathrm{Lie}}
\DeclareMathOperator{\rk}{\mathrm{rk}}
\newcommand{\Hom}{\mathrm{Hom}}
\newcommand{\End}{\mathrm{End}}
\newcommand{\Ext}{\mathrm{Ext}}

\newcommand{\bI}{{\bf I}}
\newcommand{\switt}{s_{\mathrm{Witt}}}

\newcommand{\longlabelmap}[1]{{\,\buildrel #1\over\longrightarrow\,}}
\newcommand{\longmap}{{\,\longrightarrow\,}}

\newcommand{\beqar}{\begin{eqnarray*}}
\newcommand{\eeqar}{\end{eqnarray*}}

\newcommand{\mff}{\mfrak{f}}

\nc{\bx}{\mathbf{x}}
\nc{\by}{\mathbf{y}}
\nc{\bz}{\mathbf{z}}
\nc{\ba}{\mathbf{a}}
\nc{\Fp}{\tilde{F}}
\nc{\Rp}{\tilde{R}}
\nc{\mlow}{m_{\mathrm{l}}}
\nc{\mup}{m_{\mathrm{u}}}
\nc{\ord}{\mb{ord }}
\nc{\bXp}{\bX_{\mathrm{prim}}}
\nc{\bPsi}{\mathbf{\Psi}}
\nc{\mult}{\mathrm{mult}}
\nc{\mbB}{\mathbbm{B}}
\nc{\mfor}[1]{{#1}^{\mathrm{for}}}
\nc{\Hdr}{\bH^1_{\mathrm{dR}}}
\nc{\Hcr}{\bH^1_{\mathrm{cris}}}
\nc{\Fc}{F_{\mathrm{cris}}}
\nc {\Hd}{\bH_{\d}}
\nc{\mn}{[m]n}

\nc{\tF}{\tilde{F}}
\nc{\fra}{\mfrak{f}}
\nc{\bt}{{\bf t}}
\nc{\Nn}{N^{\mn}}
\nc{\del}{\Delta}
\nc{\tilW}{\tilde{W}}
\nc{\bo}{{\bf b}}

\nc{\Di}[2]{\Delta^{#1}i^*\d^{#2}}
\nc{\di}[1]{i^*\d^{#1}}
\nc{\ue}{u_e}
\nc{\hP}{\hat{\bPsi}}

\newcommand{\Iso}{FIso}

\newcommand{\ts}{\widetilde{s}_{\mathrm{Witt}}}

\title{Delta Isocrystal and Crystalline Cohomology of abelian schemes}
\author{Sudip Pandit}
\date{}
\email{sudip.pandit@kcl.ac.uk, sudippandit20011996@gmail.com}
\address{King's College London, Strand Campus, United Kingdom}


\begin{document}

\maketitle

\begin{abstract}

We provide an explicit description of the smallest filtered sub-isocrystal generated by the Hodge filtered piece of the crystalline cohomology for an abelian scheme over a $p$-adic ring. Our method is based on the theory of arithmetic jet spaces and delta characters associated to the abelian scheme, introduced by Buium and studied later by Borger and Saha with a functor of points approach. In particular, we prove that the delta isocrystal constructed by Borger and Saha is indeed isomorphic to the fundamental smallest sub-isocrystal of the crystalline cohomology in the category of filtered $F$-isocrystals.  As an application, we establish a comparison isomorphism between the delta isocrystal and the crystalline cohomology of abelian schemes, which is governed by the group of order $1$ delta characters of the abelian scheme. \end{abstract}
\tableofcontents
\section{Introduction}
Let $p$ be a prime. Let $R$ be a $p$-adically complete discrete valuation ring with perfect residue field $k$ of characteristic $p,$ and  $K$ be the field of fractions.  Given an abelian scheme $A$ over $R,$ we denote $A_0:= A \times_{\Spec R} \Spec k,$  the special fiber over $k$.  The de Rham--crystalline comparison isomorphism $\Hcr(A_0)_K \simeq \Hdr(A)_K$ (cf.~\cite{BO}) associates a natural filtered $F$-isocrystal 
$$\mathrm{\Iso}(\Hcr(A)_{K}) := (\Hcr(A_0)_{K}, \Fc,
\Hcr(A)^\bullet_K),$$ 
where $\Hcr(A)^\bullet_K$ is the Hodge filtration given by $$\Hcr(A_0)_K \simeq \Hdr(A)_K
\supset H^0(A,\Omega_A)_K \supset \{0\}$$ and $\Fc$ is the crystalline Frobenius
operator on $\Hcr(A_0)_K$. 
This is a fundamental object in $p$-adic Hodge theory, particularly, in the study of $p$-adic Galois representations (cf.~\cite{BriCon}). In this article, we address the following fundamental questions:
\begin{itemize}
\item[(1)] What is the smallest sub-object of $\mathrm{\Iso}(\Hcr(A)_{K})$ containing $H^0(A,\Omega_A)_K$ in the category of filtered $F$-isocrystals~?
\vspace{0.2cm}

\item[(2)] Is there any elementary description of $H^0(A,\Omega_A)_K\cap \Fc (H^0(A,\Omega_A)_K)?$ \vspace{0.2cm}
\item[(3)] When $\mathrm{\Iso}(\Hcr(A)_{K})$ is generated by the invariant differentials $H^0(A,\Omega_A)_K$ in the category of filtered $F$-isocrystals~?

\end{itemize}
We provide an explicit answer to $(1)$ [cf. Corollary \ref{cor-intro}] and $(2)$ [cf. Corollary \ref{cor-intro-0}] using the theory of arithmetic jets of $A,$ and also provides a sufficient criterion for $(3)$ [cf. Theorem \ref{intro-main-3}] in terms of vanishing of the group of additive characters of the first jet $J^1A,$ which is known to hold for ``sufficiently generic" class of abelian schemes.

Buium developed the theory of arithmetic jet spaces in the category of $p$-formal schemes over $R$ in a series of articles \cite{bui95, bui97, bui00, bui-book, bui-new} in the realm of delta geometry. The theory was further explored in several articles including \cite{barc, BL, buium-miller-0, buium-miller, hurl}. Most notably in \cite{bui95}, Buium introduced the additive characters of arithmetic jet spaces of a smooth commutative group scheme $G$, which are known as the delta characters of $G.$ The study of these delta characters was motivated by their function-field companions called the differential characters (also known as the Manin characters) studied earlier by Manin and then by Buium, which led to the proof of Mordell--Lang conjecture over function fields \cite{bui92, buiHbook, Manin63}. 

Following a similar principle, the arithmetic jet spaces and delta characters found explicit Diophantine applications over number fields using $p$-adic methods. To mention a few selectively,  Buium proved the Manin--Mumford conjecture for curves, obtaining an explicit bound in \cite{bui96}. Buium--Poonen proved the finiteness of an unlikely intersection in Modular--Elliptic correspondences in \cite{BP09}. Recently, the author and Dogra obtained an explicit Mordell--Lang bound for curves in \cite{DP25}, and
Miller--Morrow obtained an explicit prime-to-$p$ Manin--Mumford bound for higher-dimensional varieties in \cite{MM25}.

On the other hand, Borger developed the geometry of Witt vectors and algebraic jet spaces as a functor of points in \cite{bor11a, bor11b}. Later, the two theories of jet spaces have been reconciled by Bartapelle--Previato--Saha in \cite{bps}. Using the functor of points approach, Borger--Saha revived the study of delta characters by relating them with the $p$-adic cohomology in \cite{BS_b}. More specifically, their approach led to constructing a filtered semilinear object $(\Hd(G),\mff^*,(\Hd(G)\supset \bXp(G)\supset \{0\})$ associated to a smooth commutative group scheme $G$ (cf.~section \ref{delta-iso} for the construction). Afterwards, the author and Saha have further developed the delta geometric foundation, pinning down the explicit relation between the arithmetic jets and their kernels, and established a comparison isomorphism between $\Hd(A)_K$ and the crystalline cohomology $\Hcr(A_0)_K$,  for an elliptic curve $A$ over $\ZZ_p$ in \cite{PS-2}. Subsequently in \cite{GPS}, the author with Gurney and Saha proved that the filtered semilinear object due to Borger--Saha is indeed a filtered $F$-isocrystal for any semiabelian scheme $G$, and extended the comparison isomorphism for  elliptic curves over finite extensions of $\Q_p.$  Such a theory has also been developed analogously in the positive characteristic context in \cite{BS_a, PS-1}.

In this article we provide an explicit description of the  delta isocrystal $\Hd(G)_K$ of a semiabelian scheme $G$ in terms of the module of primitive delta characters $\bXp(G)_K$ and its successive images under the Frobenius operator $\mff^*$. Our strategy includes establishing the compatibility between the crystalline Frobenius on the universal vectorial extension and the Frobenius on the first jet space of $A$. This will lead to the Frobenius compatibility between the delta isocrystal and the first crystalline cohomology of an abelian scheme $A$,
showing that the delta isocrystal is the smallest sub-object of the crystalline cohomology containing $H^0(A,\Omega_A)$ in the category of filtered $F$-isocrystals. This clarifies a heuristic question raised by Buium on a precise relation between the group of delta characters and the crystalline cohomology (cf.~Remark $4.6$ in \cite{bui95}). As applications, we prove comparison isomorphisms between $\Hd(A)_K$ and $\Hcr(A_0)_K$ for an abelian scheme $A$ in the category of filtered $F$-isocrystals.

 Delta geometry is intimately related to the theory of prismatic cohomology due to Bhatt--Scholze  \cite{bhsch}.
In a subsequent work, the connection between the delta isocrystal and the first prismatic cohomology of an abelian scheme would be made explicit. 
We will now state our main results following the notations in section \ref{notation}.

\subsection{Main results}Define the filtered subspaces of $\Hd(G)_K$ as
\begin{align*}\cF_{0}&= \bXp(G)_K\\
                                     \cF_{i+1}&=\bXp(G)_{K}+\mff^{*} \cF_i; ~ \mathrm{for ~}i\geq 0.
\end{align*}
Let $m_u$ be the upper splitting number of $G$ as defined in section \ref{delta-char}. Our first main result provides an explicit description of the delta isocrystal for a semiabelian scheme.
\begin{theorem}\label{intro-main-1} Let $G$ be a $\pi$-formal semiabelian scheme over $R$. Then $$\bH_{\d}(G)_K=\cF_{m_u-1}.$$ 
\end{theorem}
\noindent We have proved the above theorem as Theorem \ref{delta-crys-main}, which plays a crucial role in showing the Frobenius compatibility between the delta isocrystal and the crystalline cohomology in Theorem \ref{intro-main-2}, and essentially allows us to reduce proving only the Frobenius compatibility between the first jet space $J^1A$ and the universal vectorial extension $E(A)$ of an abelian scheme $A.$ 
 
 Borger--Saha showed  in \cite{BS_b}  that the canonical exact sequence associated to the delta isocrystal of an abelian scheme $A$ admits a map to the Hodge sequence associated to the de Rham cohomology of $A$ as below:
\begin{align*}
\xymatrix{
	0 \ar[r] & \bXp(A) \ar[d]^\Upsilon \ar[r] & \Hd(A) \ar[d]^\Phi \ar[r] & \Ext(A,\hG) \ar[d]^{\simeq}  \\
	0 \ar[r] & H^0(A,\Omega_A) \ar[r] &\bH^1_{\mathrm{dR}}(A) \ar[r] & H^1(A,\Ou_A) \ar[r] & 0.
	}
\end{align*}
Furthermore, by Theorem $1.1$ in \cite{GPS}, it is known that the map $\Phi$ is injective. Our next result is:

\begin{theorem}\label{intro-main-2} Let $A$ be a $\pi$-formal abelian scheme over $R.$ Then the injective map $\Phi:\Hd(A)_K\to \Hdr(A)_K$ is Frobenius equivariant and a map of filtered $F$-isocrystals. In other words, $$\Phi\circ \mff^*=\Fc\circ\Phi.$$
\end{theorem}
\noindent We have proved the above theorem as Theorem \ref{main-thm}. Let $\bX_1(A):=\Hom_R(J^1A,\hG)$ denote the group of delta characters of first order. Following \cite{GPS}, there is a $K$-linear isomorphism (cf. Theorem \ref{GPS-1} for the definition) $$\Upsilon: \bXp(A)_K\to H^0(A,\Omega_A)_K.$$ The restriction of the above map to $\bX_1(A)_K\hookrightarrow \bXp(A)_K$ gives us a very interesting result:

\begin{theorem}\label{cor-intro-0} Let $A$ be a $\pi$-formal abelian scheme over $R.$ Then the following map is a $K$-linear isomorphism
$$\Upsilon:\bX_1(A)_K\xrightarrow{\simeq} H^0(A,\Omega_A)_K\cap\Fc (H^0(A,\Omega_A)_K).$$
\end{theorem}

The above Theorem has been proved as Corollary \ref{cor-int}. Given an abelian scheme $A$ over $R,$  the theory of  arithmetic jets associates the  filtered $F$-isocrystal 
$$
\mathrm{\Iso}(\bH_\d(A)_{K}) := (\bH_\d(A)_{K}, \fra^*, 
\bH_\d(A)_{K}^\bullet),
$$
where $\bH_\d(A)_{K}^\bullet$ is the filtration given by 
$\bH_\d(A)_{K} \supset \bXp(A)_{K} \supset \{0\}$.


As a consequence of Theorem \ref{intro-main-1} and Theorem \ref{intro-main-2} we can draw the following result:
\begin{corollary}\label{cor-intro} Let $A$ be a $\pi$-formal abelian scheme over $R.$ Then $\mathrm{\Iso}(\bH_\d(A)_{K})$ is the smallest sub-object of $\mathrm{\Iso}(\Hcr(A)_{K})$ containing $H^0(A,\Omega_A)$ in the category of filtered $F$-isocrystals.
\end{corollary}
\noindent  The above result in particular implies that the smallest sub-object containing the Hodge piece of the crystalline cohomology of an abelian scheme has a character theoretic interpretation. We are not aware if the sub-object $\mathrm{\Iso}(\bH_\d(A)_{K})$ can also be described explicitly using the slope decomposition of $\Hcr(A_0)_{K}$ in general.  \\

\noindent An abelian scheme $A$ over $R$ is said to have a {\it Canonical Lift (CL)}
if there exists a $\beta \in \End(A)$ such that $\beta \mod \pi$ is the
absolute $q$-power Frobenius on $A \otimes_R k$. We say $A$ is CL if it acquires such a lift; otherwise, it is called non-CL.

When $A$ is CL, it follows that $\Fc$ preserves the subspace $H^0(A,\Omega_A)_K$ and we have a canonical sub-object of $\mathrm{\Iso}(\Hcr(A)_{K})$ given by $$\mathrm{\Iso}(H^0(A,\Omega_A)_{K}):=(H^0(A,\Omega_A)_K, \Fc, (H^0(A,\Omega_A)_K\supset H^0(A,\Omega_A)_K\supset  \{0\})).$$ 
 As an application of Corollary \ref{cor-intro}, we derive the following comparison isomorphisms for abelian schemes:
 \begin{theorem}\label{intro-main-3} Let $A$ be a $\pi$-formal abelian scheme over $R$. 
\begin{itemize} 
\item [(1)] If $\bX_1(A)=\{0\}$ then $$\mathrm{\Iso}(\bH_\d(A)_{K})\simeq \mathrm{\Iso}(\Hcr(A)_{K})$$ in the category of filtered $F$-isocrystals. \vspace{0.2cm}
\item[(2)] If $A$ is CL then $\dim_K \bX_1(A)_K=g$ and $$\mathrm{\Iso}(\bH_\d(A)_{K})\simeq \mathrm{\Iso}(H^0(A,\Omega_A)_{K})$$
\end{itemize}
 \end{theorem}
\noindent The above theorem has been proved as Theorem \ref{ab-sur-thm}.  It is slightly unclear at the moment whether the condition $\bX_1(A)=\{0\}$ is also necessary for the comparison isomorphism to hold in part $(1).$ We would like to remark that an elliptic curve $A$ is non-CL if and only if $\bX_1(A)=\{0\}.$ Thus Theorem \ref{intro-main-3} generalizes the previous comparison results between the delta isocrystal and the first crystalline cohomology for elliptic curves in \cite{PS-2, GPS}. Moreover, combining Theorem \ref{intro-main-1} and Corollary \ref{cor-intro} with Buium's Theorem B in \cite{bui95}, we get the following comparison isomorphisms for ordinary abelian schemes in terms of their Serre--Tate parameters:
\begin{theorem}\label{intro-main-4} Assume $\pi=p\neq2.$ Let $A$ be a $\pi$-formal abelian scheme over $R$ having ordinary closed fiber $A_0$ over $k.$ Let $q_{ij}(A)\in 1+pR, 1\leq i,j\leq g$ be the Serre--Tate parameters. Then
\begin{itemize}
\item[(1)] Assume $\det((q_{ij}(A)-1)/p)\in R^{\times}.$ Then we have 
$$\mathrm{\Iso}(\bH_\d(A)_{K})\simeq \mathrm{\Iso}(\Hcr(A)_{K})$$ in the category of filtered $F$-isocrystals.\vspace{0.2cm}
\item[(2)] Assume $q_{ij}(A)=1$ for all $i,j$, then $A$ is CL and we have 
$$\mathrm{\Iso}(\bH_\d(A)_{K})\simeq \mathrm{\Iso}(H^0(A,\Omega_A)_{K})$$
\end{itemize}
 
\end{theorem}
\noindent  Let $r_\d(A)$ be the rank of the delta isocrystal as an $R$-module. Then we have $g\leq r_\d(A)\leq 2g$ and any number in the range can be realised as the {\it delta rank} of an abelian scheme (cf. Remark \ref{delta-rk}). Moreover, it is an open question if $r_\d(A)<2g$ then $A$ has always a CL isogeny factor or not.  It would be wishful to think whether the delta rank $r_\d(A)$ leads to any interesting stratification in the moduli space of abelian schemes in the spirit of \cite{bui93, bui95-b}.

 Finally we would like to remark that the condition $\det((q_{ij}(A)-1)/p)\in R^{\times}$ is satisfied for a generic choice of the Serre-Tate parameters. So this condition may be viewed as a ``sufficiently general". By comparing the Kodiara-Spencer classes, Buium proved that in this case we do have $\bX_1(A)=\{0\}$ (cf. page 335 in \cite{bui95}). Hence the assumption in part (1) of Theorem \ref{intro-main-3} and Theorem \ref{intro-main-4} is satisfied for sufficiently large classes of abelian schemes. \\

\subsection{Organization of the paper}
In section \ref{S1} we review basic foundation on delta geometry, introduce delta isocrystal and recall important structure theorems known previously. In section \ref{S2}, we build up necessary ingredients and prove Theorem \ref{intro-main-1}. In section \ref{S3}, we have established the explicit connection between the first jet and the universal vectorial extension of abelian schemes. In section \ref{S4}, we use all the results to prove Theorem \ref{intro-main-2} and derive Theorem \ref{cor-intro-0} as a corollary . Finally in section \ref{S5}, we prove the comparison results, namely Corollary \ref{cor-intro}, Theorems \ref{intro-main-3} and  \ref{intro-main-4}.

{\bf Acknowledgements.} The author would like to thank Netan Dogra and Arnab Saha for several enlightening discussions. He also thanks to  Lance Gurney and Rajat Kumar Mishra for helpful conversations related to this topic. This article is inspired by the theory of delta characters pioneered by Alexandru Buium \cite{bui95}, followed by the fundamental work relating delta characters with the $p$-adic cohomology theory by James Borger and Arnab Saha  \cite{BS_b}. The author is grateful to all of them for teaching him about delta geometry directly and indirectly. This research was supported by Royal Society research fellowship 
 grant RF$\backslash $ERE$\backslash $231161.
 \section{Notation}
\label{notation}
We collect here some notations fixed throughout the paper.
\begin{align*}
	p &= \text{a prime number} \\
	q &= \text{a power $>1$ of }p \\
	R &= \text{a $p$-adically complete discrete valuation ring} \\
	K &= \text{the fraction field of $R$} \\
	M_K &= K\otimes_R M, \text{ for any $R$-module $M$} \\
	\pi &= \text{a fixed uniformizer of } R \\
	k &= \text{the residue field of $R$, assumed to be finite of cardinality } q \\
	\phi &= \text{an endomorphism of $R$ satisfying $\phi(x) \equiv x^q \bmod (\pi) $, for all $x \in R$} \\
	S &= \Spf R\\
	\hG &= \text{$\pi$-formal additive group scheme over $\Spf R$}\\
	G &= \text{a smooth commutative $\pi$-formal group scheme over $\Spf R$}\\
A &= \text{a $\pi$-formal abelian scheme over $\Spf R$}\\
A_0 &= \text{$A\times_S \Spec k,$ the special fiber of $A$}
\end{align*}
%
%
By a \emph{$\pi$-formal scheme}, we will mean a $\pi$-adic formal scheme over $S$. A semiabelian scheme $G$ over $R$ is a smooth commutative $\pi$-formal group scheme such that $G$ is an extension of an abelian scheme $A$ by a torus $T$ over $R.$

\section{Preliminaries}\label{S1}
In this section, we briefly recall some basic preliminaries in delta geometry required for the rest of the article. For a detailed treatment, the interested readers are referred to the excellent exposition in \cite{BS_b}.
\subsection{$\pi$-derivation and delta rings}Let $\Ou$ be a Dedekind domain of characteristic $0$ and $\mfrak{p}$ a 
non-zero prime ideal with
$k$ as the residue field and $q$ be the cardinality of $k$ where $q$ is a
power of a prime $p$. Let $\pi$ be one of the uniformizers of $\mfrak{p}\Ou_{\mathfrak{p}}$.
For any $\Ou$-algebra $B$ and $B$-algebra $A$, we define a $\pi$-derivation 
$\d$ as a set-theoretic theoretic map $\d:B \map A$ that satisfies  for 
all $x,y \in B$,

$(i)~ \d(1) = 0=\d(0)$

$(ii)~ \d (x+y) = \d x + \d y + C_\pi(u(x),u(y)) $

$(iii)~ \d(xy) = u(x)^q \d y + u(y)^q \d x + \pi \d x \d y$

where $u:B \map A$ is the structure map and 
$$C_\pi(X,Y) = 
 \frac{X^q+Y^q -(X+Y)^q}{\pi}\in \Ou_p[X,Y].$$
Given such a $\pi$-derivation $\d$, define $\phi(x):= u(x)^q + \pi\d x$
which is then a ring homomorphism satisfying
$$
\phi(x) \equiv u(x)^q \bmod \mfrak{p}.
$$
We will call such a $\phi$ a {\it lift of Frobenius} with respect to $u$. $R$. A pair $(R,\d)$ is called a {\it $\d$-ring} where $R$ is a $\pi$-adically complete  $\Ou$-algebra and $\d$ is a $\pi$-derivation on it. The $\d$-rings were introduced by Joyal in \cite{joyal}, which were extensively studied by Buium in the name of $p$-derivations and led to the foundation of delta geometry. 

\subsection{Witt vectors ($\pi$-typical)} For each $n \geq 0$, let us define $B^{\phi^n}$ to be the $R$-algebra
with structure map $R \stk{\phi^n} {\map} R \stk{u}{\map} B$ and define the \emph{ghost rings}
to be the product $R$-algebras
$\Pi^n_{\phi} B = B \times B^{\phi} \times \cdots  \times B^{\phi^n}$ and
$\Pi_{\phi}^\infty B= B \times B^{\phi} \times \cdots$.
Then for all $n \geq 1$ there exists a \emph{restriction}, or \emph{truncation},
map $T_w:\Pi_{\phi}^nB \map \Pi_{\phi}^{n-1}B$ given by $T_w(w_0,\cdots,w_n)= (w_0,\cdots,w_{n-1})$.
We also have the left shift \emph{Frobenius} operators $F_w:\Pi_{\phi}^n B \map \Pi_{\phi}^{n-1} B$ given by
$F_w(w_0,\dots,w_n) = (w_1,\dots,w_n)$. Note that $T_w$ is an $R$-algebra morphism, but
$F_w$ lies over the Frobenius endomorphism $\phi$ of $R$.

Now, as sets define
        \begin{equation}
        \label{eq-witt-coord}
        W_n(B)=B^{n+1}, 
        \end{equation}
and define the set map $w:W_n(B) \map \Pi_{\phi}^n B$  by $w(x_0,\dots,x_n)= (w_0,\dots,w_n)$ where
        \begin{equation}
        \label{eq-witt-poly}
        w_i = x_0^{q^i}+ \pi x_1^{q^{i-1}}+ \cdots + \pi^i x_i
        \end{equation}
are the \emph{Witt polynomials}.
The map $w$ is known as the {\it ghost} map. We can then define the ring $W_n(B)$, the ring
of truncated $\pi$-typical Witt vectors, by the following well-known theorem:

\begin{theorem}
\label{wittdef}
For each $n \geq 0$, there exists a unique functorial $R$-algebra structure on $W_n(B)$ such that
$w$ becomes a natural transformation of functors of $R$-algebras.
\end{theorem}
The Witt vectors admit two natural ring maps: the truncation map $T: W_n(B)
\map W_{n-1}(B)$ given by 
$$
T(a_0,\dots, a_n) = (a_0,\dots, a_{n-1}),
$$
and the Frobenius map $F : W_n(B) \map W_{n-1}(B)$ given by 
$$
F(a_0,\dots, a_n) = (a_0^q + \pi a_1,\dots ).
$$
It is the unique map which is functorial in $B$ and makes the
following diagram commutative
        \begin{equation}
        \xymatrix{
        W_n(B) \ar[r]^w \ar[d]_F & \Pi^n_{\phi} B \ar[d]^{F_w} \\
        W_{n-1}(B) \ar[r]_-w & \Pi_{\phi}^{n-1} B^n
        } \label{F}
        \end{equation}
As with the ghost components, $T$ is an $R$-algebra map but $F$ lies over the Frobenius endomorphism $\phi$
of $R$.

\subsection{Prolongations and arithmetic jets} 
Let $X$ and $Y$ be $\pi$-formal schemes over $S$. We say a pair $(u,\d)$ is a 
{\it prolongation} and we write $Y\stk{(u,\d)}{\map} X$, if $u:Y \map X$ is a 
map of $\pi$-formal $S$-schemes and $\d: \Ou_X \map u_*\Ou_Y$ is a 
$\pi$-derivation making the following diagram commute:
$$\xymatrix{
R \ar[r] & u_*\Ou_Y \\
R \ar[u]^\d \ar[r] &\Ou_X \ar[u]_\d
}
$$Following Buium~\cite{bui00} (page 103), a {\it prolongation sequence} is a chain
of prolongations
	$$
	\xymatrix{
	S & T^0 \ar_-{(u,\d)}[l] & T^1 \ar_-{(u,\d)}[l] & \cdots\ar_-{(u,\d)}[l]},
	$$
where each $T^n$ is a $\pi$-formal scheme over $S$, satisfying 
$$u^* \circ \d = \d \circ u^*
$$
and $u^*$ is the pull-back morphism on the structure sheaves induced by $u$.
We will often use the 
notation $T^*$ or $\{T_n\}_{n \geq 0}$.
Note that if the  $T^n$ are flat over $S$, then having a 
$\pi$-derivation $\d$ is equivalent to having lifts of Frobenius $\phi:T^{n+1}\to T^n$.

\noindent Prolongation sequences form a category $\mcal{C}_{S^*}$
This category has a final object $S^*$ given by $S^n=\Spf R$ for all $n$, where each $u$ is the identity and
each $\d$ is the given $\pi$-derivation on $R$.


For any $\pi$-formal scheme $X$ over $S$, for all $n \geq 0$ we define the 
$n$-th jet space $J^nX$ (relative to $S$) as a functor by
	$$
	J^nX (C) :=  X(W_n(C)) = \Hom_S(\Spf(W_n(C)),X),
	$$
for any $R$-algebra $C$.
By \cite{bps,bor11b}, the functor $J^nX$ is representable
by a $\pi$-formal $S$-scheme, which we will continue to denote as $J^nX$. 
This is precisely the arithmetic jet space constructed by Buium in 
\cite{bui95}. Also, do note that if $G$ is a 
group object, then $J^nG$ is also a group object by its definition.


The truncation and Frobenius  maps induce morphisms $u$ (projection map)
and $\phi$ (Frobenius map) respectively from
$J^nX$ to $J^{n-1}X$. The system $J^*X =\{J^nX\}_{n=0}^\infty$ is known
as the {\it canonical prolongation sequence} as in proposition 1.1 in
\cite{bui00}. Then $J^*X$ satisfies the following 
universal property: 
\begin{theorem}[Proposition 1.1 in 
\cite{bui00}] \label{Uni-jet}For any $T^* \in \mcal{C}_{S^*}$ and $X$ a $\pi$-formal
scheme over 
$S^0$, we have
\begin{equation*}	
\label{canprouniv}
	\Hom(T^0,X) = \Hom_{\mcal{C}_{S^*}}(T^*, J^*X)
\end{equation*}	
\end{theorem}

\subsection{Delta characters of group schemes}\label{delta-char}
\color{black}
If $G$ is a smooth $\pi$-formal group scheme over $S$, the projection map
$u: J^nG \map G$ is a surjection of $\pi$-formal group schemes 
\cite{bor11b, bui00}. Let $N^nG$ be the kernel of the morphism $u$ and
hence we have the following canonical short exact sequence of $\pi$-formal
group schemes
\begin{align}\label{short1}
0 \map N^nG \stk{\iota}{\map} J^nG \stk{u}{\map} G \map 0.
\end{align} 
Combining Theorem $4.3$ in  \cite{BS_b} and Theorem $1.2$ in \cite{PS-2}, 
one shows that for all $n$ we have a canonical isomorphism
\begin{align}
N^nG \simeq J^{n-1}(N^1G),
\end{align}
and the associated Frobenius morphism for the canonical prolongation 
sequence $\{N^*G\}_{n=0}^\infty$, denoted $\fra: N^nG \map N^{n-1}G$ 
satisfies
\begin{align}
\label{phi-fra}
\phi^{\circ 2} \circ \iota = \phi \circ \iota \circ \fra.
\end{align}
Set $\bX_n(G):= \Hom(J^nG, \hG)$ to be the $R$-module of additive 
characters of the $n$-th jet space $J^nG$. Then $\bX_n(G)$ is called the
module of {\it delta characters} of $G$ with order $\leq n$.
Let $\bX_\infty(G)$ be the direct limit of $\bX_n(G)$ with respect to 
pulling back by $u^*$. Also pulling back a delta character $\Theta$ by $\phi^*$ 
endows $\bX_\infty(G)$ with a semilinear action of $\phi^*$ satisfying
$\phi^* \cdot r = \phi(r) \cdot \phi^*$ for all $r \in R$, see 
$(3.11)$ in \cite{BS_b}.

Let $G$ be a semiabelian scheme. For every $n$, we have the following short exact sequence
\begin{equation}
\label{se}
0 \map N^nG \map J^nG \map G \map 0
\end{equation}
Applying $\Hom(-, \hG)$ to the above short exact sequence gives us
\begin{equation}
\label{sel}
0 \map \bX_n (G)/\bX_0(G) \map \Hom(N^nG,\hG) \stk{\partial}{\map} \Ext(G,\hG) 
\end{equation}
We have for a semiabelian scheme $G$ the extension group
$\Ext(G,\hG)$ is a finite free $R$-module. 

Let $\bI_n(G) := \mb{image}(\partial)$. 
Note that since for all $n$, there are maps $\Hom(N^nG,\hG) \stk{u^*}{\inj} \Hom(N^{n+1}G,\hG)$,
we have $\bI_n(G) \subset \bI_{n+1}(G)$. Define 
\begin{equation}
\label{bigI}
\bI(G):= \varinjlim \bI_n(G)
\end{equation}
and 
$$
h_i= \rk \bI_i(G) - \rk \bI_{i-1}(G)
$$
for all $i \geq 1$.
We define the {\it upper splitting number} to be the smallest number 
$\mup\geq 1$ such that $h_n=0$ for all $n \geq \mup$.  
Note that $\mup$ exists since $\Ext(G,\hG)$ is a finite free $R$-module for a semiabelian scheme $G$ over $R,$ and  $$\bI_0(G) \subset \bI_{1}(G) \subset \cdots \subset \Ext(G,\hG).$$

\subsection{Delta isocrystal} \label{delta-iso}Let $G$ be a smooth commutative group scheme over $R$ and $\bx_0=(x_{01},\dots,x_{0g})$ be an \'{e}tale coordinate system of $G$ around the identity section. 
Recall from section $4.1$ in \cite{BS_b} that the \'etale coordinates $\bx_0$ around the identity section of $G$, gives  local Witt coordinates $(\bx_0,\ldots\bx_n)$ on $J^nG$ and correspondingly $(\bx_1,\ldots\bx_n)$ on $N^nG.$ 
This gives a basis of $\Lie(G),  \Lie(J^nG)$ and $\Lie (N^nG)$ respectively, still denoted by the same.
Let us fix such a coordinate system $\bx_0$ on $G.$
\begin{theorem}[Theorem $4.4$ in \cite{BS_b}]
\label{comfact}
Let $G$ be a smooth group scheme over $S$. Then
the morphism of group schemes $(\iota \circ \mfrak{f} - 
\phi \circ \iota) : N^nG \map J^{n-1}G$ uniquely factors through  $N^1G$ as
$$\xymatrix{
N^nG \ar[rr]^-{\iota \circ \mfrak{f} -\phi \circ \iota}  
\ar[d]_u & & J^{n-1}G \\
N^1G \ar[rru]_\Delta & &
}$$
\end{theorem}
For any delta character $\Theta \in \bX_n(G)$, consider the differential
map $$D\Theta : T_0(J^nG) \map T_0(\hG)$$ between the tangent spaces at the 
identity sections. With respect to the local Witt coordinates chosen,  let the differential map be given by
\begin{align}
\label{differential}
D\Theta = (A_0, \cdots ,A_n)
\end{align}
where $A_j \in \mb{Mat}_{1\times g} (R)$.

\begin{proposition}
\label{diff}
Let $\Theta$ be a character in $\bX_n(G)$.
\begin{enumerate}
        \item We have
                $$
               \mfrak{f}^*(\iota^*\Theta) = \iota^*\phi^*\Theta +
\gamma_{\Theta}. \bPsi_1,
                $$
                where $\gamma_{\Theta}=\pi A_0$. 
        \item For $n\geq 2$, we have
                $$
               (\mfrak{f}^{n-1})^* \iota^*\phi^*\Theta =\iota^*(\phi^{\circ n})^*\Theta.
                $$
\end{enumerate}
\end{proposition}
\begin{proof}
See Proposition $6.3$ in \cite{BS_b}. Do note that here we flipped the difference and hence sign of $\gamma_{\Theta}$ for our convenience; however, the proof follows exactly by the same argument.
\end{proof}
\noindent For any $R$-module $M$, let us denote
$$
M_{\phi} = R\otimes_{\phi,R}M.
$$
Let $G$ be any group scheme, we define the $R$-module of {\it primitive} delta characters $\bXp(G)$ as 
$$
\bXp(G):= \varinjlim \bX_n(G)/\phi^*(\bX_{n-1}(G)_\phi).
$$
A class $[\Theta] \in \bXp(G)$ represented by a delta character $\Theta$, is 
a primitive character that can not be written as a sum of a combination of delta 
characters of lower order via various pull-backs by $\phi^*$. Often, we will denote the primitive characters by $\Theta$ itself if there is no confusion. For each $n\geq 1,$ there is a natural inclusion of $R$-modules 
$\bX_n(G)/\phi^*\bX_{n-1}(G)\rightarrow \bXp(G).$ 
Recall that the $\phi$-linear map $\phi^*:\bX_{n-1}(G) \to \bX_n(G)$ induces a linear map
$\bX_{n-1}(G)_{\phi}\to \bX_n(G)$, which we will abusively also denote $\phi^*$.
We then define
$$
\bH_{n}(G) = \frac{\Hom(N^{n}G,\hG)}{i^*\phi^*(\bX_{n-1}(G)_{\phi})}.
$$
Note that 
$u:N^{n+1}G\map N^nG$ induces $u^*:\Hom(N^nG,\hG) \map \Hom(N^{n+1}G,\hG)$. 
Moreover, since $u$ commutes with both $i$ and $\phi$, we have
$$
u^*i^*\phi^*(\bX_n(G)_\phi) = i^*\phi^*u^*(\bX_n(G)_\phi) \subset i^*\phi^*(\bX_{n+1}(G)_\phi),
$$
and hence $u$ also induces a map $u^*:\bH_n(G) \map \bH_{n+1}(G)$.
Define 
\begin{equation}
	\label{bigH}
	\Hd(G)= \varinjlim \bH_n(G)
\end{equation}
where the limit is taken in the category
of $R$-modules. Similarly, $\mfrak{f}: N^{n+1}G \map N^nG$ induces 
$\mfrak{f}^*:\Hom(N^nG,\hG) \map \Hom(N^{n+1}G,\hG)$, 
which descends to a $\phi$-semilinear map of $R$-modules
\begin{equation}	
	\label{latfrobH}
	\mfrak{f}^*:\bH_n(G) \map\bH_{n+1}(G)
\end{equation}
because (by Proposition $6.3(2)$ in \cite{BS_b}) we have $$\mfrak{f}^*i^*\phi^*(\bX_{n-1}(G)_\phi)= 
i^*\phi^*\phi^*(\bX_{n-1}(G)_\phi) \subset i^*\phi^*(\bX_{n}(G)_\phi).$$
This induces a $\phi$-semilinear endomorphism $$\mfrak{f}^*:\Hd(G) \map \Hd(G).$$
 
 \subsection{Earlier known results on delta isocrystal of semiabelian schemes}
 For readers convenience, we have collected here the following important structure theorems on the delta isocrystal from \cite{ GPS, PS-2}.
 
\begin{theorem}[Theorem $1.5$ in \cite{PS-2}]\label{yeah}
Let $G$ be a $\pi$-formal semiabelian scheme of relative dimension $g$ 
over $R$. Then
we have 
\begin{enumerate}
\item The $R$-module $\bXp(G)$ is free of rank $g$.
\item The $R$-module $\Hd(G)$ is free satisfying
	$g \leq \rk_R\Hd(A) \leq 2g$.
\end{enumerate}		
\end{theorem}
\noindent The theorem above has been proved for abelian schemes, however the same proof works for semiabelian schemes without any difficulty. 

From now on wards (throughout the rest of the paper)  we consider $R = W(\bF_q)$ where $q$ is a power of our 
prime $p$ with $k = \bF_q$ as its residue field and the fixed lift of 
Frobenius $\phi$ on $R$ is taken to be the identity map. 
Hence a semilinear
operator for an isocrystal over $K$ are linear operators.

\begin{theorem} [Theorem $1.1$ in \cite{GPS}]\label{GPS-1} Let $G$ be a $\pi$-formal semiabelian scheme over $R.$ Then the map
\begin{align}\label{chars--diff} \Upsilon: \bXp(G)&\to H^0(G,\Omega_G)\\
\nonumber\Theta(\bx_0) &\to \frac{\gamma_{\Theta}}{\pi} \cdot d\bx_0
\end{align}
is injective with $\pi$-torsion cokernel. In particular, $\Upsilon:  \bXp(G)_K\to H^0(G,\Omega_G)_K$ is an isomorphism of $K$-vector spaces.
\end{theorem}
 \begin{theorem}[Theorem $1.2$ in \cite{GPS}]
\label{bijfra-1}
Let $G$ be a $\pi$-formal semiabelian  scheme of relative dimension $g$ 
over $R$. 
Then the semilinear operator $\fra^*: \bH_\d(G)_K \map \bH_\d(G)_K$ is a 
bijection.

Hence $(\bH_\d(G)_K,\fra^*, (\bH_\d(G)_K\supset \bXp(G)_K\supset \{0\} ))$ is a filtered $F$-isocrystal.

\end{theorem}
\section{An explict description of $\Hd(G)$ for semiabelian schemes} \label{S2}
In this section we show that the delta isocrystal $\Hd(G)$ of a semiabelian scheme $G$ can be explicitly described in terms of the primitive delta characters of $G$ and the Frobenius operator $\mff^*.$ More specifically, we will prove that $(\Hd(G),\mff^*)$ is the smallest isocrystal containing $\bXp(G).$
Note that for each $n\geq1$, we have the following commutative cartesian diagram of $R$-modules
\begin{align}\label{com-1}
\xymatrix{
 \bX_n(G) \ar@{^-->}[d]^{p_n} \ar@{^{(}->}[r]^-{i^*} & \Hom(N^nG,\hG) \ar@{^-->}[d]^{\tilde{p}_n}\\
    \bXp(G) \ar@{^{(}->}[r]^-{i^*}& \Hd(G) 
}
\end{align}
\begin{lemma}\label{com-2} Let $G$ be a group scheme. Then, with the above notation, we have
$$i^*p_n\bX_n(G)=i^*\bXp(G) \cap \tilde{p}_n\Hom(N^nG,\hG).$$ In other words, we have $$i^*\bX_n(G)/i^*\phi^*\bX_{n-1}G=i^*\bXp(G)\cap \bH_n(G).$$

\end{lemma}
\begin{proof} We will only prove $i^*p_n\bX_n(G)=i^*\bXp(G) \cap \tilde{p}_n\Hom(N^nG,\hG)$, and the latter equality in the lemma follows by the definition. By the commutativity of the diagram we have $i^*p_n(\bX_n(G))\subset i^*\bXp(G) \cap \tilde{p}_n\Hom(N^nG,\hG).$ On the other hand, let $[\Theta]\in i^*\bXp(G) \cap \tilde{p}_n\Hom(N^nG,\hG).$ Then $[\Theta]=[\Psi],$ for some $\Psi\in  \Hom(N^nG,\hG)$ and $[\Theta]=[\Theta_m]$ for some $\Theta_m\in \bX_m(G)$ and $m\in\mathbb{N}.$ If $m\leq n,$ then $[\Theta]\in p_n\bX_m(G)\subset p_n\bX_n(G)$ and we are done.

 Let $m>n$, by the commutativity of the diagram \ref{com-1}, we get $$i^*\Theta_m-\Psi= i^*\phi^*\Theta'\in i^*\bX_m(G).$$ Hence $i^*\Theta_m- i^*\phi^*\Theta'=\Psi \in i^*\bX_n(G)\subset i^*\bX_m(G).$  As $i^*$  is injective, we have $\Theta_m-\phi^*\Theta' \in \bX_n(G)$. Then $[\Theta]=[\Theta_m-\phi^*\Theta']\in p_n\bX_n(G)$  and that completes the proof.
\end{proof}

\begin{theorem}\label{order-1} Let $G$ be a  $\pi$-formal semiabelian scheme over $R$ of relative 
dimension $g$. Then $$i^*\bX_1(G)_K/i^*\phi^{*}\bX_0(G)_K=i^*\bXp(G)_K\cap \mff^{*}i^*\bXp(G)_K$$
\end{theorem}
\begin{proof} Let $G$ be any commutative group scheme and $$i^*\Theta\in i^*\bXp(G)_K\cap \mff^{*}i^*\bXp(G)_K.$$ Then there is a $\Theta'\in \bXp(G)_K$ such that $i^{*}\Theta=\mff^{*}i^{*}\Theta'.$ 
On the other hand by part $(1)$ of Proposition \ref{diff}, we get $\mff^{*}i^{*}\Theta'=\gamma_{\Theta'}\cdot \bPsi_1$ in $\Hd(G)_K.$  Hence $i^*\Theta\in \bH_1(G)_K.$ Then by diagram \eqref{com-1} we get that $i^*\Theta\in i^* \bX_1(G)_K/i^*\phi^*\bX_0(G)_K$.

For the other direction, let $i^*\Theta \in i^*\bX_1(G)_K/i^*\phi^{*}\bX_0(G)_K$. Then $i^{*}\Theta\in \bH_1(G)_K.$ Since $\bPsi_1$ forms a $K$-basis of $\bH_1(G)_K,$ there exists an $\alpha\in K^n$ such that $i^{*}\Theta=\alpha \cdot \bPsi_{1}.$ If $G$ is a semiabelian scheme then by \eqref{chars--diff}, we have the map $$\Upsilon: \bXp(G)_K \map H^0(G,\Omega_G)_K$$ is an isomorphism of $K$-vector spaces. 
In particular, fixing an \'{e}tale coordinate $\bf{x}_0$ of $G$, there is a $[\Theta']\in \bXp(G)_K$ such that $\Upsilon(\Theta')=\frac{1}{\pi}\alpha \cdot d\bf{x}_0.$ But by definition $\Upsilon_{K}(\Theta') =\frac{\gamma_{\Theta'}}{\pi} \cdot d\bf{x}_0$ which implies that $\gamma_{\Theta'}=\alpha$. Hence $$\mff^{*}i^{*}\Theta'=\gamma_{\Theta'}\cdot \bPsi_1=\alpha \cdot \bPsi_1=i^{*}\Theta,$$ showing that $i^*\Theta\in i^*\bXp(G)_K\cap \mff^{*}i^*\bXp(G)_K$, and we are done.
\end{proof}
\begin{lemma} \label{frob-bij-2}The map $\mff^{*}:i^*\bXp(G)_K\to \bH_1(G)_K$ is an isomorphism of $K$-vector spaces. In particular, $\dim_K \bXp(G)_K=\dim _K \bH_1(G)_K=g.$
\begin{proof} Note that Theorem \ref{bijfra-1} implies $\mff^*$ restricted to $i^*\bXp(G)_K$ is injective. Since we have  $\mff^{*}i^*\bXp(G)_K\subset \bH_1(G)_K.$ To prove that $\mff^*$ is also surjective, it is enough to show that $\dim_K \bH_1(G)_K=g.$ Indeed, $\bH_1(G)_K$ being a quotient vector space of $\Hom(N^1G,\hG)_K,$  we get $$g=\dim_K \bXp(G)_K\leq \dim _K \bH_1(G)_K\leq  \dim_K \Hom(N^1G,\hG)_K=g.$$
Therefore $\mff^*:i^*\bXp(G)_K\to \bH_1(G)_K$ is a bijection. 
\end{proof}
\end{lemma}
\begin{corollary}\label{sum} If $\bX_1(G)_{K}=\{0\}$ then $$\bH_{\d}(G)_{K}=i^*\bXp(G)_K+\mff^{*}i^*\bXp(G)_K.$$
\end{corollary}
\begin{proof} By the definition, we have $i^*\bXp(G)_K+\mff^{*}i^*\bXp(G)_K\subset \bH_{\d}(G)_{K}$. To prove the equality, we will show that they have the same dimension as  K-vector spaces. Note that $\dim \bH_{\d}(G)_{K} \leq 2g$ and $\dim \bXp(G)_K=g$.  Moreover, since $\bX_1(G)_K=\{0\}$, by Theorem \ref{order-1} we have $$i^*\bXp(G)_K\cap \mff^{*}i^*\bXp(G)_K=i^*\bX_1(G)_K/i^*\phi^{*}\bX_0(G)_K=\{0\}.$$ Following part $(1)$ of Proposition \ref{diff}, the operator $\mff^{*}$ sends any element $[\Theta]\in \bXp(G)_K$ to $\gamma_{\Theta}.[\bPsi_{1}]\in \bH_1(G)_K\subset \Hd(G)_K.$ Then using Lemma \ref{frob-bij-2}, one can compute $$\dim_K (i^*\bXp(G)_K+\mff^{*}i^*\bXp(G)_K)=2g=\dim_K \bH_{\d}(G)_{K}$$ which completes the proof. 
\end{proof}

\subsection{Filtered subspaces of the delta isocrystal} Let us define the following filtered subspaces $\cF_{i}$ of $\bH_{\d}(G)_K$ inductively: \begin{align*}\cF_{0}&= i^*\bXp(G)_K\\
                                     \cF_{i+1}&=i^*\bXp(G)_{K}+\mff^{*} \cF_i; ~\mathrm{for~} i\geq 0
\end{align*}
\begin{lemma} \label{step-1}We have $$ \bI_{1}(G)_{K}\simeq \bH_1(G)_K/(i^*\bX_{1}(G)_K/i^*\phi^*\bX_0(G)_K)\simeq \cF_1/\cF_0.$$
\end{lemma}
\begin{proof} The left isomorphism follows from the definition of $\bI_{1}(G)$ and the first isomorphism theorem of the short exact sequence. Hence, it is enough to prove the right isomorphism. 

By Lemma \ref{frob-bij-2}, we know $\mff^{*}:i^*\bXp(G)_K\longrightarrow \bH_{1}(G)_{K}$ is a $K$-linear isomorphism. Then consider the inclusion map $$j:\bH_1(G)_{K}\longrightarrow i^*\bXp(G)_{K}+\mff^{*}i^*\bXp(G)_{K}$$ by identifying $\bH_1(G)_{K}$ with $\mff^{*}i^*\bXp(G)_K$.  Then, composing with the quotient map, we get 
$$\overline{j}:\bH_1(G)_{K}\to \cF_1/\cF_{0}.$$
By Theorem \ref{order-1} we get   $\ker(\overline{j})=i^*\bX_1(G)_K/i^*\phi^{*}\bX_0(G)_K.$ Then $\overline{j}$ descends to the required isomorphism.
\end{proof}

\begin{lemma}\label{frob-shift} For each $n\geq 1$, we have $\bH_{1}(G)_{K}+\mff^{*}\bH_{n}(G)_{K}=\bH_{n+1}(G)_{K}.$
\end{lemma}
\begin{proof} Note that by Proposition $5.2$ in \cite{BS_b}, we have $$\Hom(N^1G,\hG)+\mff^*\Hom(N^nG,\hG)=\Hom(N^{n+1}G,\hG).$$ Then, passing to the respective quotient maps, we can derive the proof of the lemma. 
\end{proof}
\begin{lemma}\label{frob-n} The image of the map $\mff^{*}: \cF_{n}\longrightarrow \bH_{\d}(G)_{K}$ is $\bH_{n+1}(G)_{K}$. In particular, $\mff^*:  \cF_{n}\to \bH_{n+1}(G)_{K}$ is an isomorphism of $K$-vector spaces.
\end{lemma}
\begin{proof} We will show this by induction on $n$. The case $n=0$ is true by Lemma \ref{frob-bij-2}. Let us assume the assertion is true for $n-1$, and we will show it is true for $n$. By assumption, we have $\mff^{*}\cF_{n-1}=\bH_{n}(G)_{K}$. 
Therefore, we have
\begin{align*}\mff^*\cF_n&=\mff^*i^*\bXp(G)_K+\mff^*(\mff^*\cF_{n-1})~[\mathrm{by~ the ~definition~ of}~ \cF_n]\\
&= \mff^*i^*\bXp(G)_K+\mff^*\bH_n(G)_K ~ [\mathrm{by~ induction ~hypothesis}]\\
&=\bH_1(G)_K+\mff^*\bH_n(G)_K ~[ \mathrm{Lemma} ~\ref{frob-bij-2}]\\
&=\bH_{n+1}(G)_K. ~[ \mathrm{Lemma} ~\ref{frob-shift}]
\end{align*} 
This completes the proof. 
\end{proof}
\begin{lemma}\label{order-n} We have $$i^*\bX_{n+1}(G)_{K}/i^*\phi^{*}\bX_{n}(G)_{K}=i^*\bXp(G)_{K} \cap \mff^{*}\cF_{n}$$
\end{lemma}
\begin{proof} By Lemma \ref{frob-n}, we get $$i^*\bXp(G)_{K} \cap \mff^{*}\cF_{n}= i^*\bXp(G)_{K} \cap \bH_{n+1}(G)_{K}.$$ 

On the other hand by Lemma \ref{com-2}, we get $$i^*\bXp(G)_{K} \cap \bH_{n+1}(G)_{K}=i^*\bX_{n+1}(G)_{K}/i^*\phi^{*}\bX_{n}(G)_{K}.$$ Combining the two equalities, we are done.
\end{proof}

\begin{lemma}\label{short-n} For any $n\geq 0$, we have $$ \bI_{n+1}(G)_{K}\simeq \bH_{n+1}(G)_K/ (i^*\bX_{n+1}(G)_K/i^*\phi^{*} \bX_{n}(G)_{K})\simeq \cF_{n+1}/ \cF_0.$$
\end{lemma}
\begin{proof} Define the inclusion map $j: \bH_{n+1}(G)_K\to \cF_{n+1}$ by identifying $\bH_{n+1}(G)_K$ with $\mff^*\cF_n$ under the isomorphism in Lemma \ref{frob-n}. Then composing with the quotient  we get the map $$\overline{j}: \bH_{n+1}(G)_K\to \cF_{n+1}/\cF_{0}.$$ Note that $\ker(\overline{j})=i^*\bX_{n+1}(G)_K/i^*\phi^{*} \bX_{n}(G)_{K}$  by Lemma \ref{order-n}. Therefore $\overline{j}$ descends to the required isomorphism.
\end{proof}

\begin{theorem}\label{delta-crys-main} Let $G$ be a $\pi$-formal semiabelian scheme then $$\bH_{\d}(G)_K=\cF_{m_u-1},$$ where $m_u$ is the upper splitting number of $G$.
\end{theorem}
\begin{proof} For each $n\geq 1$, Lemma \ref{short-n} gives rise to the following inclusion between the short exact sequences
\begin{align*}
\xymatrix{
 0\ar[r] & \bXp(G)_K \ar@{=}[d] \ar[r] &\cF_n  \ar@{^{(}->}[d]\ar[r]& \bI_{n}(G)_K\ar@{^{(}->}[d]\ar[r] &0 \\
  0\ar[r]& \bXp(G)_K \ar[r]& \Hd(G)_K \ar[r]& \bI(G)_K \ar[r] &0
}
\end{align*}
Then by the definition of upper splitting number from Section $7$ in \cite{BS_b}, we have $\bI_{m_u-1}(G)_K=\bI(G)_K,$ which implies $\bH_{\d}(G)_K=\cF_{m_u-1},$ following the diagram above.
\end{proof}
\noindent As a corollary, we obtain a generalisation of a result  known for abelian schemes due to Buium (cf. part $(1), (2)$ of Proposition $3.2$ in \cite{bui95}).
\begin{corollary}\label{semiab-CL} Let $G$ be a semiabelian scheme over $R$. 
\begin{itemize}
\item[(1)]Then $G$ is CL if and only if $\bXp(G)=\bX_1(G)_K/\phi^{*}\bX_0(G)_K.$ \vspace{0.2cm}
\item[(2)] If $\bX_1(G)_K=\{0\}$, then $\bXp(G)_K=\bX_2(G)_K.$
\end{itemize}
\end{corollary}

\begin{proof} Part $(1)$ follows from Theorem \ref{order-1}  and the fact that $G$ is CL if and only if the exact sequence of group schemes $$0\map N^1G\map J^1G\map G\map 0$$
splits.

 For part (2), note that $\bX_1(G)_K=\{0\}$ implies $m_l\geq2$. Also, by Theorem \ref{delta-crys-main} and Corollary \ref{sum}, we have $m_u=2$. Therefore $\bXp(G)_K=\bX_2(G)_K.$
\end{proof}
In what follows we will understand $\bXp(G)$ is an $R$-submodule of $\Hd(G)$ without mentioning the inclusion map $i^*.$

\section{Universal vectorial extension and first jet space of abelian schemes}\label{S3}
Let $A$ be a $\pi$-formal abelian scheme over $R,$ which will be fixed throughout. In this section, we establish a precise relation between the first jet space $J^1A$ of an abelian scheme $A$ and its universal vectorial extension $E(A).$ We also show that the Frobenius $\phi:J^1A\to A$ is compatible with the crystalline Frobenius on $E(A).$
\subsection{Universal vectorial extension}Let $A$ be an abelian scheme over $R.$ Then the theory of the universal vectorial extensions following \cite{MazMess} gives a short exact sequence of group schemes
\[ 0\to V(A)\to E(A)\to A\to 0
\]
where $V(A)=\bH^1(A,\cO_A)^{\vee}$ and $E(A)$ is the universal vectorial extension of $A.$ Moreover, it is known that  $\Lie (E(A))^{\vee}\simeq \Hdr(A),$ and the following exact sequence at the level of dual Lie algebras is isomorphic to the Hodge sequence associated to $A$ (cf. page 328 in \cite{Raynaud83b})

 
\[
\xymatrix{
	0 \ar[r] &  \Lie(A)^{\vee}\ar[d]^-{\simeq} \ar[r] &  \Lie(E(A))^{\vee} \ar[d]^-{\simeq} \ar[r] &\Lie(V(A))^{\vee}\ar[d]^-{\simeq} \ar[r] & 0 \\
	0 \ar[r] & H^0(A,\Omega_A) \ar[r] &\bH^1_{\mathrm{dR}}(A) \ar[r] & H^1(A,\Ou_A) \ar[r] & 0.
	}
	\]
We will use these identifications in the rest of the article without mentioning repeatedly.
\subsubsection{Crystalline structure} Let $F:A_0\to A_0$ be the $q$-th power Frobenius. The crystalline structure of universal vectorial extension gives rise to an absolute  endomorphism 
\[\EE(F): E(A)\to E(A)
\]
such that the following diagram commutes (cf. page 337 in \cite{Raynaud83b})
\begin{align}\label{frob-univ}\xymatrix{
E(A_0)\ar[d]^{E(F)}\ar@{=}[r]&E(A)_0\ar[d]^{\EE(F)_0}\ar[r]& A_0\ar[d]^F\\
E(A_0)\ar@{=}[r]&E(A)_0\ar[r] & A_0
}
\end{align}
Moreover, under the identification $\Lie(E(A))^{\vee}\simeq \Hdr(A)$, the absolute endomorphism $\EE(F)$ induces the crystalline Frobenius $\Fc:\Hdr(A)\to \Hdr(A)$ (cf. page 330 in \cite{Raynaud83b}).

\subsection{Frobenius compatibility of $E(A)$ and $J^1A$} 
 In what follows we abbreviate the notation $N^1A$ as $N^1$. The commutative diagram \eqref{frob-univ} shows that the composition map $$E(A) \xrightarrow{\EE(F)} E(A) \xrightarrow{u_e} A$$ is a lift of the Frobenius $E(A_0)\xrightarrow{\ue} A_0\xrightarrow{F}A_0$ relative to the projection map $u_e: E(A)\to A.$ 
Since $R$ is $\pi$-torsion free and $E(A)$ is smooth over $R,$ the sheaf of $R$-modules $\Ou_{E(A)}$ is also $\pi$-torsion free. Consequently, the two maps $$(u_e, u_e\circ\EE(F)): E(A)\to A$$ induce a prolongation. Therefore by the universal property (cf. Theorem \ref{Uni-jet}) of the first jet space $J^1A$, we have  a unique map of prolongations $$\xi: E(A)\to J^1A.$$ In other words, we have  the following diagram 

\begin{align}\label{jet-univ-1}\xymatrix{
0\ar[r]& V(A)\ar[d]^{\chi}\ar[r]^{i_e}&E(A)\ar[d]^{\xi}\ar[r]^{u_e}& A\ar@{=}[d]\ar[r]&0\\
0\ar[r]& N^1\ar[r]&J^1A\ar[r]^{u} & A\ar[r]&0
}
\end{align}
commutes and $u_e\circ\EE(F)=\phi \circ \xi.$

Given a character $\Psi\in \Hom_R(A,\hG)$ we get the following exact sequence by push-forward
 $$\xymatrix{0\ar[r]& N^1\ar[d]^{\Psi}\ar[r]&J^1A\ar[d]^{g_{\Psi}}\ar[r]^{u} & A\ar@{=}[d]\ar[r]&0\\
 0\ar[r]& \hG^g\ar[r]&A^*_{\Psi}\ar[r] & A\ar[r]&0
}
 $$
 On the other hand, the universal property of $E(A)$ implies that: there is a unique map $\chi_{\Psi}: V(A)\to \hG$ such that $A^*_{\Psi}$ is the push forward of $E(A)$ by $\chi_{\Psi}.$ Therefore we have,
\begin{align}\label{jet-univ-2}\xymatrix{
0\ar[r]& V(A)\ar[d]^{\chi_{\Psi}}\ar[r]&E(A)\ar[d]^{\xi_{\Psi}}\ar[r]& A\ar@{=}[d]\ar[r]&0\\
0\ar[r]& \hG\ar[r]&A^*_{\Psi}\ar[r] & A\ar[r]&0
}
\end{align}

\begin{lemma} \label{frob-jet-univ} Let $\Psi\in \Hom_R(N^1,\hG)$ be a character. Then the following  two maps coincide $$g_{\Psi}\circ\xi=\xi_{\Psi}: E(A)\to A^*_{\Psi}.$$ 
\end{lemma}
\begin{proof} Let $f,g : E(A) \rightarrow A^*_{\Psi}$ be two group scheme maps over $A.$ Then consider the difference $f-g.$ The image of $f-g$ lands in the kernel $ \hG.$ Hence $f-g \in \mathrm{Hom}(E(A),\hG).$ 

On the other hand applying $\mathrm{Hom}(- , \hG)$ to $0 \rightarrow V  \rightarrow E(A)  \rightarrow A \rightarrow 0 $ gives us the long exact sequence: 
$$0\rightarrow \mathrm{Hom}(E(A), \hG) \rightarrow \mathrm{Hom}(V, \hG) \stackrel{\Delta}{\rightarrow} \mathrm{Ext} (A,  \hG)\rightarrow \ldots $$

But the connecting map $\Delta$ is an isomorphism and hence $\mathrm{Hom}(E(A), \hG) =\{0\}$ which implies that $f=g.$
\end{proof}

\begin{corollary}\label{cor-frob-univ} Let $\Psi\in \Hom_R(N^1,\hG)$ be a character. The following two maps are equal $$\chi_{\Psi}=\Psi\circ\chi: V(A)\to\hG.$$ 
\end{corollary}
\begin{proof} Directly follows from Lemma \ref{frob-jet-univ}.
\end{proof}
\begin{corollary} \label{frob-eq} Let $A$ be an abelian scheme over $R.$ Then the following two maps coincide $$d\xi\circ d\phi=\Fc:\Lie(A)^{\vee}\to \Lie (E(A))^\vee \simeq\Hdr(A).$$
\begin{proof} By diagram \ref{jet-univ-1}, we get $$\phi\circ \xi=u_e \circ\EE(F): E(A)\to A.$$ The induced map in differentials gives the desired equality, completing the proof.
\end{proof}
\end{corollary}
\section{Frobenius compatibility between $\Hd(A)$ and $\Hcr(A)$}\label{S4}
\subsection{The map $\Phi$}
Following section $6.2$ in \cite{BS_b}, let us recall Borger--Saha's construction of the map $\Phi$ from the character group of kernels to the de Rham cohomology of $A$. For each $n,$ consider the jet sequence 
\[0\to N^n\to J^nA\to A\to 0
\]
Given an element $\Psi\in \Hom(N^n,\hG),$ we have the push-forward sequence 
\begin{align}\label{jet-push}\xymatrix{
0\ar[r]& N^n\ar[d]^{\Psi}\ar[r]&J^nA\ar[d]^{g_\Psi}\ar[r]& A\ar@{=}[d]\ar[r]&0\\
0\ar[r]& \hG\ar[r]&A^*_{\Psi}\ar[r] & A\ar[r]&0
}
\end{align}
This gives rise to the connecting homomorphism in the long exact sequence after applying the $\Hom(-,\hG)$ functor to the jet sequence. Therefore we have  
\begin{align}\label{connecting} \partial:\Hom(N^n,\hG) &\to \Ext(A,\hG)\\
\nonumber \Psi&\mapsto [A^*_\Psi].
\end{align}
Based on the choice of local \'{e}tale coordinates $\bx_0$ for $A$, Borger--Saha shows that there is a canonical splitting of the Lie algebras of the jet sequence.
 We let it to be $$\switt: \Lie J^nA 
\map \Lie N^n $$  given by $\switt (\bx_0, \cdots , \bx_n)= (\bx_1,\cdots,
\bx_n).$
Thus, we have the following split exact sequence of $R$-modules
	$$
	\xymatrix{
	0 \ar[r] & \Lie N^n \ar[r]^-{Di} & \Lie J^nA \ar@/^/[l]^{{\switt}}
	\ar[r]^-{Du} & \Lie(A) \ar[r] & 0
	}
	$$
Let $v: \Lie A \map \Lie J^nA$ be the corresponding section, satisfying
$\switt= \mathbbm{1} - v \circ Du$. Then in Witt coordinates, $v$ is given by $v(\bx_0) = (\bx_0,0,
\cdots, 0)$.
Let $s_\Psi$ denote the induced splitting of the push-out extension
	\begin{align}\label{lie-split}
	\xymatrix{
	0 \ar[r] & \Lie \hG \ar[r] & \Lie(A^*_\Psi) \ar@/^/[l]^{s_\Psi}
	\ar[r] & \Lie(A) \ar[r] & 0
	}
	\end{align}
It can be described explicitly in terms of the composition
	$$
	\tilde{s}_\Psi : \Lie J^nA \times \Lie\hG \longmap \frac{\Lie  J^nA \times \Lie  \hG}{\Lie \Gamma(N^n)} 
	\longlabelmap{s_\Psi} \Lie \hG
	$$
by
	\begin{align}\label{splitting}
	\tilde{s}_\Psi(\bx,y)= D\Psi({\switt}(\bx)) + y.
	\end{align}
Recall that $\Ext^\sharp(A,\hG)$ parametrizes isomorphism classes of 
extensions of $A$ by $\hG$, along with a splitting of the corresponding 
short exact sequence of the (commutative) Lie algebras. (See~\cite{MazMess}, p.\ 13--14, where it
would be denoted $\mathrm{Extrig}(A,\hG)$.) 	
We then have a morphism of exact sequences
	\begin{equation}
	\label{trivext}
	\xymatrix{
		0 \ar[r] & \bX_n(A) \ar[r] \ar[d] & 
		\HomA(N^n,\hG) \ar[r]^{\partial} \ar[d]_{\Psi\mapsto (A_\Psi^*,s_\Psi)}& 
		\Ext(A,\hG)  \ar@{=}[d]& \\
		0 \ar[r] &\dualmod{\Lie(A)} \ar[r] & \Ext^\sharp(A,\hG) \ar[r] & 
\Ext(A,\hG) \ar[r] & 0
	}
	\end{equation}
Then by proposition $6.1$(3) in \cite{BS_b}, $i^*\phi^*(\bX_{n-1}(A)_{\phi})$ is in 
the kernel
of the middle vertical map. Therefore 
this diagram induces a diagram
	\begin{equation}
	\label{diag-crys}
	\xymatrix{
	0 \ar[r] & \frac{\bX_n(A)}{\phi^*(\bX_{n-1}(A))} \ar[r] 
\ar[d]_-\Upsilon & 
	\bH_n(A) \ar[r] \ar[d]^-{\Phi} & 
	\bI_n(A) \ar[r] \ar@{^{(}->}[d]& 0\\
	0 \ar[r] &\dualmod{\Lie(A)} \ar[r] & \Ext^\sharp(A,\hG) \ar[r] & \Ext(A,\hG)
	 \ar[r] & 0
	}
	\end{equation}
where $\bI_n(A)$ denotes the image of $\partial:\Hom(N^n,\hG) \map \Ext_A(A,\hG)$.
Passing to the limit gives a diagram:
\begin{equation}
	\label{diag-crys-limit}
	\xymatrix{
	0 \ar[r] & \bXp(A) \ar[d]_\Upsilon \ar[r] & 
	\Hd(A) \ar[d]_\Phi \ar[r] &\bI(A) \ar@{^{(}->}[d] \ar[r] &  0 \\
	0 \ar[r] & \dualmod{\Lie (A)} \ar[r] & \Ext^\sharp(A,\hG) \ar[r] & \Ext(A,\hG)
	 \ar[r] & 0 
	}
\end{equation}
The $R$-module $\Ext^\sharp(A,\hG)$ is canonically isomorphic to the de Rham cohomology $\Hdr(A)$ of the abelian scheme
$A$ (cf. section $6.1$ in \cite{BS_b}).
\subsection{Reinterpretation of $\Phi$ at level $1$} For $n=1$, we can reinterpret the above diagram as follows. Consider the commutative diagram \eqref{jet-univ-2} at the level of Lie algebras 
\begin{align}\label{lie-jet-univ}\xymatrix{
0\ar[r]& \Lie V(A)\ar[d]^{D\chi}\ar[r]&\Lie E(A)\ar[d]^{D\xi}\ar[r]& \Lie A\ar@{=}[d]\ar[r]&0\\
0\ar[r]& \Lie N^1\ar[r]&\Lie J^1A\ar[r] &\Lie  A\ar[r]&0
}
\end{align}
 Since all the $R$-modules are free, applying $\Hom(-,\hG)$, we get

\begin{align}\label{dual-lie-jet-univ}\xymatrix{
&0\ar[r]&\bX_1(A)\ar@{-->}[lddd]^{\Upsilon}\ar[r]&\Hom(N^1,\hG)\ar@{-->}[lddd]^-{\Phi}\ar[ld]_-{\Phi'}\ar[d]^{\Psi\to D\Psi}\ar[r]^-{\partial}&\Ext(A,\hG)\ar@{-->}[lddd]^{\alpha} \\
0\ar[r]&\Lie  (A)^{\vee} \ar@{=}[dd]\ar[r]&\Lie (J^1A)^{\vee}\ar[dd]^-{d\xi} \ar[r] &\Lie (N^1)^{\vee}\ar@/_/[l]_-{{\ts}} \ar[dd]_{
d\chi}\ar[r]&0\\
& & & &\\
0\ar[r]&\Lie (A)^{\vee}\ar[r]&\Lie (E(A))^{\vee}\ar[r]&  \Lie (V(A))^{\vee}\ar[r]&0
}
\end{align}
where $\ts$ is the dual of the splitting $\switt.$ To prove the commutativity of the above diagram, it is enough to show that \begin{align}\label{com-eq} d\chi(D\Psi)=\alpha(\partial(\Psi)).\end{align}
Note that the derivative map $D: V(A)^{\vee}\rightarrow \Lie(V(A))^{\vee}$ given by $\psi\mapsto D\psi$ is an isomorphism of group schemes. Also, the connecting map $\Delta:V(A)^{\vee}\rightarrow \Ext(A,\hG)$ defined by the push-forward of the universal vectorial extension by an element of $V(A)^{\vee}$ is an isomorphism. Combining both these isomorphisms, in diagram \eqref{dual-lie-jet-univ} we identify $$\alpha:\Ext(A,\hG)\xrightarrow{\simeq} \Lie(V(A))^{\vee}$$ and we have $\alpha\circ \Delta=D.$

On the other hand, following \eqref{connecting}, the connecting map $\partial$ can also be described as the push-forward of the jet extension by a character. Therefore by Corollary \ref{cor-frob-univ}, for any $\Psi\in \Hom(N^1,\hG)$ we have $$\Delta(\Psi \circ \chi)=\Delta(\chi_{\Psi})=\partial(\Psi).$$
Thus we obtain $$ d\chi(D\Psi)=D(\Psi\circ \chi)=\alpha \circ \Delta (\Psi\circ \chi)=\alpha(\partial(\Psi))$$  proving the required equality \eqref{com-eq}. 
 Let us now consider the isomorphism of the following two exact sequences of $R$-modules
\begin{align}\label{ext-lie}\xymatrix{
0\ar[r]&\Lie (A)^{\vee}\ar[r]\ar@{=}[d]&\Lie (E(A))^{\vee} \ar[d]^-{\simeq}_{\alpha'}\ar[r]^-{di_e}&\Lie (V(A))^{\vee}\ar[r]\ar[d]^-{\simeq}_{\alpha^{-1}}&0\\
0 \ar[r] &\dualmod{\Lie(A)} \ar[r] &\Ext^\sharp(A,\hG)  \ar[r] & \Ext(A,\hG) 
  \ar[r] & 0.
}
\end{align}
\subsubsection{The map $\alpha'$} The map $\alpha'$ can be given explicitly. Let the element $\alpha^{-1}(di_e(\eta))\in \Ext(A,\hG)$ gives an extension class represented by $$0\to \hG\to E\to A\to 0,$$ which is the push-forward of the universal vectorial extension by $g\in V(A)^{\vee}$ such that $Dg=di_e(\eta)$ as follows:
\begin{align}\label{univ-push}\xymatrix{
0\ar[r]& V(A)\ar[d]^{g}\ar[r]^{i_e}&E(A)\ar[d]\ar[r]& A\ar@{=}[d]\ar[r]&0\\
0\ar[r]& \hG\ar[r]&E\ar[r] & A\ar[r]&0
}
\end{align}
Then note that $E=\frac{E(A)\times \hG}{\Gamma(V(A))}$ where $\Gamma(V(A)):=\{i_e(v),-g(v)\in E(A)\times \hG\}.$ Thus we have a map $\tilde{s_\eta}:\Lie(E(A))\times \hG\to \hG$ given by $(\bx,y)\to \eta(\bx)+y.$ Furthermore, $Dg=di_e(\eta)$ implies that $\tilde{s_\eta}$ restricted to $\Lie(\Gamma(V(A)))$ is trivial. Hence it induces a map $s_\eta: \Lie(E)\to \hG$, which gives a splitting as follows:
\begin{align}\label{lie-split-2}
	\xymatrix{
	0 \ar[r] & \Lie \hG \ar[r] & \Lie(E) \ar@/^/[l]^{s_\eta}
	\ar[r] & \Lie(A) \ar[r] & 0
	}
	\end{align}
Then the map $\alpha'$ is explicitly given by 
\begin{align}\label{alpha'}\alpha'(\eta)=(\alpha^{-1}(di_e(\eta)), s_\eta).
\end{align}
The following proposition shows that under the isomorphism \eqref{ext-lie}, diagram \eqref{dual-lie-jet-univ} recovers the map $\Phi$ constructed by Borger--Saha at level $1.$
 
\begin{proposition} Let $\Phi'$ be as in diagram \eqref{dual-lie-jet-univ}. Then  $$\alpha'\circ d\xi\circ\Phi'= \Phi|_{\Hom(N^1,\hG)}.$$
\end{proposition} 
\begin{proof} Recall the map $\Phi$ is constructed as follows
 \begin{align*} \Phi: \Hom(N^1,\hG)&\to \Ext(A,\hG)\to \Ext^\sharp(A,\hG)\\
 \Psi&\map\partial (\Psi)\map (\partial (\Psi), s_{\Psi})  
 \end{align*}
 Note that by \eqref{com-eq} we have  $d\chi (D\Psi)=\alpha(\partial (\Psi)).$ Moreover, by \eqref{splitting} the following is also true $$\ts (D\Psi(\bx))=D\Psi(\switt(\bx))= \tilde{s}_{\Psi}:\Lie(J^1A)\to \hG.$$
Recall the above induces the splitting $s_{\Psi}$ associated to $\Psi$ as in \eqref{lie-split}. 
Then following the description of the map $\alpha'$ as in \eqref{alpha'}, we see that $$\alpha'(d\xi\circ \Phi')=(\partial(\Psi),s_{\Psi})=\Phi(\Psi)$$ This completes the proof.
 \end{proof}
\noindent In what follows, under the identifications $\Hdr(A)\simeq \Lie(E(A))^{\vee}\simeq \Ext^{\sharp}(A,\hG),$ we will assume 
\begin{align}\label{level-one-Phi}
\Phi=d\xi\circ\Phi': \Hom(N^1,\hG)_K\to \Hdr(A)_K.
\end{align}
\subsection{Compatibility between the two Frobini $\mff^*$ and $\Fc$}
 Recall that by Theorem \ref{GPS-1}, we have the following inclusion of exact sequences of $R$-modules
\begin{align}\label{delta-hodge}
\xymatrix{
	0 \ar[r] & \bXp(A)\ar@{^{(}->}[d]^\Upsilon \ar[r] & \Hd(A) \ar@{^{(}->}[d]^\Phi \ar[r] & \bI(A) \ar@{^{(}->}[d] \ar[r] & 0 \\
	0 \ar[r] & H^0(A,\Omega_A) \ar[r] &\bH^1_{\mathrm{dR}}(A) \ar[r] & H^1(A,\Ou_A) \ar[r] & 0.
	}
\end{align}
In this subsection, we will show that
\begin{theorem}\label{main-thm} Let $A$ be a $\pi$-formal abelian scheme over $R.$ Then $$\Phi\circ\mff^*=\Fc\circ\Phi:\Hd(A)_K\to\Hdr(A)_K.$$ 
\end{theorem}
Note that by Theorem \ref{delta-crys-main}, we have $\Hd(A)_K$ is generated by $\bXp(A)_K$ and their successive images applying $\mff^*$. Thus we can reduce Theorem \ref{main-thm} to the following theorem:
\begin{theorem}\label{Frob-main-thm}Let $A$ be a $\pi$-formal abelian scheme over $R.$ Then $$\Phi\circ\mff^*=\Fc\circ\Upsilon:\bXp(A)_K\to \Hdr(A)_K.$$
\end{theorem}
\subsubsection{Fundamental characters} Given a smooth commutative group scheme $G$ of relative dimension $g$ over $R,$ following section $5$ in \cite{GPS}, the map $$\phi \circ i: N^1 G \xrightarrow{i} J^1G\xrightarrow{\phi} G$$ factors through the subgroup scheme $G(\pi)\subset G,$ and it is given by $\bx \mapsto \pi \bx$ in coordinates. 
Then by Proposition $4.1$ in \cite{bui-new}, the $g$-tuple of additive characters $\bPsi_1\in \Hom_R( N^1G, \hG^g)_K$ can be defined as
 
\begin{align}
\label{Psi1}
\bPsi_1(\bx) := \frac{1}{\pi}\log_G \circ (\phi \circ i) (\bx) = \frac{1}{\pi}
\log_G(\pi \bx)
\end{align}
Following the same Proposition, we note that the character $\bPsi_1$ is defined over $R$ if the absolute ramification index $e(R)\leq p-1.$ We call it the {\it fundamental characters}, which forms a $K$-basis of $\Hom_R( N^1G, \hG)_K.$ 

\noindent We also have the {\it integral fundamental characters} $$\hP=\pi^v\bPsi_1: N^1\to \hG^g$$ where $v ~(v=0~ \mathrm{if} ~e(R)\leq p-1)$ is an integer as defined in Proposition $4.1$ in \cite{bui-new}. The integral fundamental characters forms an $R$-basis of $\Hom_R(N^1G,\hG).$

The rest of this section is devoted to the proof of Theorem \ref{Frob-main-thm}. Recall from section $4.1$ in \cite{BS_b} that the \'etale coordinates $\bx_0=(x_{01},\dots,x_{0g})$ around the identity section of $A$  fixes a basis $d\bx_0$ of $H^0(A,\Omega_A)\cong\Lie(A)^{\vee}.$ Let us also fix a basis $B=\{\Theta_1,\ldots, \Theta_g\}$ of $\bXp(A).$ 

\begin{lemma}\label{Frob-up} Let $[\Upsilon]^{d\bx_0}_{B}$ and $[\mff^*]^{\bPsi_1}_{B}$ be the matrices of the $K$-linear maps $\Upsilon$ and $\mff^*$ with respect to the prescribed basis respectively. Then $$[\mff^*]^{\bPsi_1}_{B}=\pi\cdot[\Upsilon]^{d\bx_0}_{B}.$$
\end{lemma}
\begin{proof} Following Theorem 8.2 in \cite{BS_b} the matrix $[\Upsilon]^{d\bx_0}_{B}$ can be explicitly given by $[A_{01},\ldots, A_{0g}]$ as in \eqref{differential}. Then the proof follows from Proposition \ref{diff} part (1).
\end{proof}

The \'etale coordinate system of $A$ gives a basis on $\Lie (A)^{\vee},\Lie (J^1A)^{\vee}$ and  $\Lie(N^1)^{\vee}$. With respect to these bases, we get two matrices $\ts(D\bPsi_1),d\phi\in M_{2g\times g}(K).$
\begin{lemma} \label{Frob-psi}  We have $$\ts(D\bPsi_1(\bx_1))= \frac{1}{\pi} d\phi(d\bx_0) \in \Lie(J^1A)_K^{\vee}.$$
\end{lemma}
\begin{proof} Note that the derivative of the $\log_G$ is the identity map of Lie algebras. The proof follows from the definition of $\bPsi_1$ and the fact that $di$ is exactly the map $\ts$ on differentials.
\end{proof}
\noindent \textbf{Proof of Theorem \ref{Frob-main-thm}.} We will show that the following diagram commutes.
\begin{align}\label{Frob-com-main}\xymatrix{
\bXp(A)_K\ar[d]^{\Upsilon}\ar[r]^-{\mff^*}&\Hom(N^1,\hG)_K\ar[d]^{\Phi'}\ar[rd]^-{\Psi\to D\Psi}\\
\Lie  (A)^{\vee}_K \ar[rd]^-{\Fc}\ar[r]^-{d\phi}&\Lie (J^1A)^{\vee}_K\ar[d]^{d\xi} \ar[r] &\Lie (N^1)^{\vee}_K \ar@/^/[l]^-{{\ts}}\\
&\Lie (E(A))^{\vee} _K&
}
\end{align}
Fix a system of \'etale coordinates $\bx_0$ on $A.$ By Lemma \ref{frob-bij-2}, we have $$\mff^*:\bXp(A)_K\to \Hom(N^1,\hG)_K.$$ By the commutative diagram \ref{dual-lie-jet-univ}, we get
\begin{align}
\nonumber \Phi(\bPsi_1(\bx_1))=&=d\xi(\Phi'(\bPsi_1(\bx_1)))~ [\mathrm{by}~ \eqref{level-one-Phi}]\\
 \label{eq-1}
                                               &=d\xi(\ts(D\bPsi_1(\bx_1)))~[\mathrm{by ~commutative~diagram}~ \ref{dual-lie-jet-univ}]\\
\nonumber                             & =d\xi(1/\pi \cdot d\phi(d\bx_0))~ [\mathrm{Lemma}~ \ref{Frob-psi}]\\
\nonumber                             &=1/\pi\cdot d\xi\circ d\phi(d\bx_0) \\
                              &=1/\pi\cdot \Fc(d\bx_0)~ [\mathrm{Corollary}~ \ref{frob-eq}]
\end{align}
Then fixing a basis $B=\{\Theta_1,\ldots,\Theta_g\}$ of primitive characters in $\bXp(A)$, we have 
\begin{align*}
\Phi(\mff^*(B))&=[\mff^*]^{\bPsi_1}_{B}\cdot\Phi(\bPsi_1(\bx_1))\\
                       &=\pi\cdot [\Upsilon]^{d\bx_0}_{B} \cdot\Phi(\bPsi_1(\bx_1))~ [\mathrm{Lemma}~ \ref{Frob-up}]\\
                      &=\pi\cdot [\Upsilon]^{d\bx_0}_{B}\cdot 1/\pi\cdot \Fc(d\bx_0)~[\mathrm{by}~ \eqref{eq-1}]\\
                      &=\Fc(\Upsilon(B))
\end{align*}
which completes the proof.\qed
 
 As an application we have a very interesting result describing the largest $\Fc$-stable subspace of $H^0(A,\Omega_A)_K$ in terms of order $1$ delta characters.
 \begin{corollary}\label{cor-int} Let $A$ be a $\pi$-formal abelian scheme over $R.$ Then the following map is a $K$-linear isomorphism
$$\Upsilon:\bX_1(A)_K\xrightarrow{\simeq} H^0(A,\Omega_A)_K\cap\Fc (H^0(A,\Omega_A)_K).$$
\end{corollary}
\begin{proof} The proof directly follows by combining Theorem \ref{order-1}, Theorem \ref{GPS-1} and Theorem \ref{Frob-main-thm}.
\end{proof}

 \section{Comparison results between $\Hd(A)$ and $\Hcr(A)$} \label{S5}
 As a corollary of Theorem \ref{delta-crys-main} and Theorem \ref{main-thm}, we obtain an explicit characterization of the delta isocrystal as a sub-object of the crystalline cohomology.
 \begin{corollary}\label{main-cor} Let $A$ be an abelian scheme over $R.$ Then $\mathrm{\Iso}(\bH_\d(A)_{K})$ is the smallest sub-object of $\mathrm{\Iso}(\Hcr(A)_{K})$ containing $H^0(A,\Omega_A)$ in the category of filtered $F$-isocrystals.
\end{corollary}

 As an application, we deduce the comparison isomorphisms for abelian schemes as follows:
 
 \begin{theorem}\label{ab-sur-thm}Let $A$ be a $\pi$-formal abelian scheme over $R$. 
\begin{itemize} 
\item [(1)] If $\bX_1(A)=\{0\}$ then $$\mathrm{\Iso}(\bH_\d(A)_{K})\simeq \mathrm{\Iso}(\Hcr(A)_{K})$$ in the category of filtered $F$-isocrystals. \vspace{0.2cm}
\item[(2)] If $A$ is CL then $\dim_K \bX_1(A)_K=g$ and $$\mathrm{\Iso}(\bH_\d(A)_{K})\simeq \mathrm{\Iso}(H^0(A,\Omega_A)_{K})$$
where $\mathrm{\Iso}(H^0(A,\Omega_A)_{K})$ is the sub-object of $\mathrm{\Iso}(\Hcr(A)_{K}).$
\end{itemize}
\end{theorem}
\begin{proof} (1) Since $\bX_1(A)=\{0\},$ by part $(2)$ of Corollary \ref{semiab-CL},  we have $$\bXp(A)_K=\bX_2(A)_K.$$ Hence $m_u=2.$ Then by Corollary \ref{sum}, we have $$\Hd(A)_K=i^*\bXp(A)_K\oplus \mff^*i^*\bXp(A)_K.$$ Therefore $\dim_K \Hd(A)_K=2g,$ and we are done by Corollary \ref{main-cor}.\\

\noindent (2) If $A$ is CL then by part (1) of Corollary \ref{semiab-CL}, we have $\bXp(A)_K=\bX_1(A)_K.$
Hence $m_u=1$ and by Theorem \ref{delta-crys-main}, we have $$\Hd(A)_K=\cF_0=\bXp(A)_K.$$ Therefore $\dim_K \Hd(A)_K=g,$ and we are done by Corollary \ref{main-cor}.
\end{proof}
We would like to remark that the above theorem, in particular, recovers the earlier comparison results between the delta isocrystal and the first crystalline cohomology for elliptic curves in \cite{GPS, PS-2} as a special case. Indeed, by section $9$ in \cite{BS_b} the following holds:

 For an elliptic curve $A$ over $R,$ the $R$-module $\bX_1(A)=\{0\}$ if and only if $A$ is non-CL if and only if $\dim_K \Hd(A)_K=2.$
Thus Theorem \ref{ab-sur-thm}  implies the following theorem:
\begin{theorem}[Theorem 1.3 in \cite{GPS}]\label{Iso-crys-11}
Let $A$ be an elliptic curve over $R$. 

(1) If $A$ is a non-CL elliptic curve then
$$
\mathrm{\Iso} (\Hd(A)_{K}) \simeq \mathrm{\Iso} (\Hcr(A)_{K})
$$
in the category of filtered $F$-isocrystals.

(2) If $A$ is CL then $$\mathrm{\Iso} (\Hd(A)_{K}) \simeq 
\mathrm{\Iso} (H^0(A, \Omega_A)_{K})$$ 
in the category of filtered $F$-isocrystals 
where $\mathrm{\Iso} (H^0(A,\Omega_A)_{K})$
is the one dimensional sub-object of $\mathrm{\Iso}(\Hcr(A)_{K})$.
\end{theorem}
\begin{remark}\label{delta-rk}  Given any two positive integers $r$ and $g$ such that $g\leq r\leq 2g$ there is an abelian scheme $A$ with $\dim A=g$ and $\rk \Hd(A)=r.$

 Indeed, let $A=A_1^{r-g}\times A_2^{2g-r}$, where $A_1$  and $A_2$ are two elliptic curves over $R$ such that $A_1$ is non-CL and $A_2$ is CL. Then note that $$\dim A=(r-g)+(2g-r)=g.$$ On the other hand, by Theorem \ref{Iso-crys-11} we have $$\rk \Hd(A)=2(r-g)+(2g-r)=r.$$
\end{remark}
Moreover, combining Corollary \ref{main-cor} with Buium's Theorem B in \cite{bui95}, we can directly derive the following comparison isomorphisms for ordinary abelian schemes in terms of their Serre--Tate parameters:
\begin{theorem} Assume $\pi=p\neq2.$ Let $A$ be a $\pi$-formal abelian scheme over $R$ having ordinary closed fiber $A_0$ over $k.$ Let $q_{ij}(A)\in 1+pR, 1\leq i,j\leq g$ be the Serre--Tate parameters. Then
\begin{itemize}
\item[(1)] Assume $\det((q_{ij}(A)-1)/p)\in R^{\times}.$ Then we have 
$$\mathrm{\Iso}(\bH_\d(A)_{K})\simeq \mathrm{\Iso}(\Hcr(A)_{K})$$ in the category of filtered $F$-isocrystals.\vspace{0.2cm}
\item[(2)] Assume $q_{ij}(A)=1$ for all $i,j$, then $A$ is CL and we have 
$$\mathrm{\Iso}(\bH_\d(A)_{K})\simeq \mathrm{\Iso}(H^0(A,\Omega_A)_{K})$$
\end{itemize}
\end{theorem}
\begin{proof} Theorem B in \cite{bui95} implies that $\bXp(A)_K=\bX_2(A)_K$ in the first case, and $\bXp(A)_K=\bX_1(A)_K$ in the second case. Therefore we are done by Corollary \ref{main-cor} as in the proof of Theorem \ref{ab-sur-thm}.
\end{proof}

\footnotesize{

\end{document}